\documentclass[a4paper]{amsart}

\usepackage{amssymb,amsmath,amsthm,amsfonts}
\usepackage{bbm}
\usepackage{enumerate}
\usepackage{mathtools,url,graphicx}

\usepackage{doi}

\hypersetup{
	colorlinks=true,       
	linkcolor=magenta,          
	citecolor=blue,        
	filecolor=violet,      
	urlcolor=orange          
}

\usepackage{bookmark}
\usepackage{hyperref}
\hypersetup{pdfstartview={FitH}}

\usepackage{tikz}
\usepackage{float, caption, subcaption}
\usetikzlibrary{positioning,intersections, calc, arrows, decorations.markings, math}

\newcommand{\eps}{\varepsilon}
\newcommand{\R}{\mathbb{R}}
\newcommand{\C}{\mathbb{C}}
\newcommand{\N}{\mathbb{N}}

\DeclarePairedDelimiter\abs{\lvert}{\rvert}
\DeclarePairedDelimiter\norm{\lVert}{\rVert}

\newcommand{\longcomment}[1]{}

\newtheorem{theorem}{Theorem}
\newtheorem{lemma}[theorem]{Lemma}
\newtheorem{proposition}[theorem]{Proposition}

\newtheorem{problem}[theorem]{Problem}
\newtheorem*{problem*}{Problem}

\newtheorem*{conjecture*}{Conjecture}

\theoremstyle{definition}

\numberwithin{equation}{section}

\def\d{\mathrm{d}}
\def\supp{\mathrm{supp}}
\def\one{\mathbbm{1}}

\def\low{\mathrm{low}}
\def\high{\mathrm{high}}

\begin{document}

\title[Extension estimates on a strip]{Fourier extension estimates on a strip in $\mathbb{R}^2$}

\author[A. Bulj]{Aleksandar Bulj}
\address{Aleksandar Bulj\\
        Department of Mathematics, Faculty of Science\\
        University of Zagreb\\
        Bijeni\v{c}ka cesta 30\\
        10000 Zagreb, Croatia}
\email{aleksandar.bulj@math.hr}

\author[S. Shiraki]{Shobu Shiraki}
\address{Shobu Shiraki\\
        Department of Mathematics, Faculty of Science\\
        University of Zagreb\\
        Bijeni\v{c}ka cesta 30\\
        10000 Zagreb, Croatia
        }
\email{shobu.shiraki@math.hr}

\date{\today}

\subjclass[2020]{42B10}

\keywords{Fourier extension operator, Radon transform}

\begin{abstract}
Given a smooth curve with nonzero curvature $\Sigma\subset \mathbb{R}^2$, let  $E_{\Sigma}$ denote the associated Fourier extension operator.
For both general compact curves and the parabola, we characterize the pairs $(p,q)\in [1,\infty]^2$ for which the estimates
$\|E_{\Sigma}f\|_{L^q(\Omega)}\leq C\|f\|_{L^p(\Sigma)}$ and $(\mathcal{R}(|E_{\Sigma}f|^{q}))^{\frac{1}{q}}\leq C\|f\|_{L^p(\Sigma)}$ 
hold, where $\Omega$ is a strip in $\mathbb{R}^2$ and $\mathcal{R}$ denotes the Radon transform.
This work continues the study of mass concentration of $x\mapsto E_{\Sigma}f(x)$ near lines in $\mathbb{R}^2$, initiated by Bennett and Nakamura \cite{BN21} and later extended by Bennett, Nakamura, and the second author in \cite{BNS24}, where expressions of the form $(\mathcal{R}(|E_{\Sigma}f|^{2}))^{\frac{1}{2}}$ were studied.
\end{abstract}

\maketitle



\section{Introduction}

Let $\Sigma \subset \R^2$ be a $C^2$ curve. We write $E_{\Sigma}$ for the Fourier extension operator defined on Schwartz functions $f \in \mathcal{S}(\R^2)$ by
\begin{equation}
    \label{e:extension_def}
    E_{\Sigma}f(x):= \int_{\Sigma} f(\xi)e^{2\pi i \xi \cdot x} d\mu(\xi), \quad x\in \R^2,
\end{equation}
where the measure $\mu$ is defined as follows. 
When $\Sigma$ is a compact curve, $\mu$ denotes the induced Lebesgue measure (i.e., the surface measure).
On the other hand, when $\Sigma$ is the parabola
\[\mathbb{P}^1 := \{(\xi_1, \xi_1^2): \xi_1\in \R\},\]
due to the connection of the Fourier extension operator with the Schr\"odinger propagator, the measure $\mu$ denotes the pushforward measure from the horizontal variable, i.e.,
\[\int_{\mathbb{P}^1} f(\xi) d\mu(\xi)=\int_{\R}f(\xi_1,\xi_1^2)d\xi_1.\]

Given a curve $\Sigma\subset \R^2$, and $p,q \in [1,\infty]$, the existence of a constant $C>0$ such that the following estimate holds for all Schwartz functions $f\in \mathcal{S}(\R^2)$
    \begin{equation}\label{e:restriction_2d}
        \|E_{\Sigma}f\|_{L^q(\R^2)}
        \leq
        C
        \|f\|_{L^p(\Sigma, d\mu)}
    \end{equation}
is known as the \textit{Restriction theorem for $\Sigma$}.
The problem of determining all pairs of $(p,q)\in [1,\infty]^2$ for which the estimate \eqref{e:restriction_2d} holds was solved up to the endpoints by Fefferman (with attribution also to Stein) \cite{Feff70}, and at the endpoint by Zygmund \cite{Zyg74}. See the section \ref{s:preliminary} for the statement of the theorem. 

To gain information about the geometry of level sets of the function $x \mapsto E_{\Sigma}f(x)$, it is natural to consider weighted estimates.
A classical proposal in this direction is the Mizohata–Takeuchi conjecture, which asserted that for every
$C^2$ curve $\Sigma \subset\R^2$ with surface measure $\mu$ and any nonnegative weight $w:\R^2\to [0,\infty)$, 
\[\int_{\R^2} \abs{E_{\Sigma}f(x)}^2w(x)dx \leq C \norm{f}_{L^2(\Sigma)}^2\norm{Xw}_{L^{\infty}}
\]
holds, 
where $Xw$ denotes the X-ray transform of $w$. This conjecture, however, was recently disproved by Cairo \cite{Cai25}. 
Motivated by this, we study weighted restriction estimates in a related but distinct setting.

When the weight $w$ is the characteristic function of a strip in $\R^2$, the original Mizohata--Takeuchi conjecture becomes trivial, since the right-hand side of the inequality is infinite.
Because of the decay rate of the Fourier transform of the measure $\mu$ in the normal direction to the curve, one can choose a strip for which the left-hand side is also infinite. Thus, one cannot simply replace $\|Xw\|_{L^\infty}$ with a different constant to obtain a nontrivial estimate.
However, if we replace the $L^2$ norms in the conjecture with suitable $L^p$ and $L^q$ norms, we can prove nontrivial weighted estimates of the form 
\[\|E_{\Sigma}f\|_{L^q(\R^2, w(x)dx)}\leq C_{w}\norm{f}_{L^p(\Sigma)}.\]
This certainly holds for all $(p,q)$ in the range where the corresponding Fourier extension estimate (without a weight) holds. On the other hand, because we restrict the domain of integration, it is natural to expect the estimate to hold over a larger range of $(p,q)$. 

By letting the width of the strip shrink to zero in an appropriate sense, we are naturally led to consider Radon transform–type estimates, which capture how the mass of $x \mapsto E_{\Sigma}f(x)$ can concentrate near lines in $\R^2$. 
This line of investigation was initiated by Bennett and Nakamura \cite{BN21}, who studied expressions of the form $\mathcal{R}(\abs{E_{\Sigma}f}^2)$, where $\mathcal{R}$ denotes the Radon transform and $\Sigma$ is a sphere in $d\ge 2$ dimensions.
Some results were later extended by Bennett, Nakamura, and the second author to the $k$-plane transform in \cite{BNS24}.
In particular, among other results they showed the following result in two dimensions. 
Denote with $\mathbb{S}^1$ the unit circle centered at $0$
and let $n(\xi)$ denote the normal vector of $\Sigma$ at $\xi$.
For a given curve $\Sigma$, suppose that $\omega\in \mathbb{S}^1$ satisfies \textit{transversality condition}
\begin{equation}\label{e:transversal}
n(\xi)\cdot \omega \neq 0, \quad \text{for all}\quad \xi\in \Sigma, 
\end{equation}
Then,
\begin{equation}\label{e:BN}
\mathcal{R}(|E_\Sigma f|^2)
(\omega,t)
=
\int_{S}\frac{|g(\xi)|^2}{|n(\xi)\cdot\omega|}\,\d \sigma(\xi)
\end{equation}
holds for all $(\omega,t)\in \mathbb{S}^1\times \R$. Moreover, if there exists a constant $c>0$ such that $|n(\xi)\cdot \omega|\geq c$ for all $\xi\in S$, then a simple computation (see Section~\ref{s:proof_compact}) yields 
\[
\|E_\Sigma f\|_{L^2(\{|x\cdot \omega-t|\leq 1/2\})}
\leq C
\|f\|_{L^2(\Sigma)}.
\]

If the curve $\Sigma$ is compact, then by interpolation and H\"older's inequality, this already provides the full range of estimates.
The problem thus becomes more interesting when the curve is noncompact or when there exists a point $\xi\in \Sigma$ at which the transversality condition \eqref{e:transversal} fails.

We formally state the problems that we will study.

\begin{problem}
\label{prob:strip}
    Given a $C^2$ curve $\Sigma\subset \R^2$ and a vector $\omega \in \mathbb{S}^1$, determine the range of $(p,q)\in [1,\infty]^2$ for which there exists a constant $C=C_{\Sigma, \omega, p,q}>0$ such that the following estimate holds for all Schwartz functions $f\in \mathcal{S}(\R^2)$
    \begin{equation}
        \label{e:resriction_strip}
        \norm{E_{\Sigma}f}_{L^q( \{\abs{x\cdot\omega-t}\leq 1/2\})} \leq C\norm{f}_{L^p(\Sigma)}.
    \end{equation}
\end{problem}
Since we do not have any substantial insights related to the width of the strip, we fix the width to be 1. However, if we let the width tend to $0$ in appropriate way, the estimate \eqref{e:resriction_strip} reduces to a statement regarding the Radon transform $\mathcal{R}$ defined by:
\[\mathcal{R}g(\omega,t):=\int_{\omega \cdot x =t}g(x)d\lambda_{\omega, t}(x),\]
where $(\omega,t)\in (\mathbb{S}^1\times \R)$ and $d\lambda_{\omega,t}$
denotes the Lebesgue measure on the line $\{x\in \R^2: x\cdot \omega =t\}$. This leads to the following analogous question regarding the Radon transform.
\begin{problem}
\label{prob:radon}
    Given a $C^2$ curve $\Sigma\subset \R^2$ and a vector $\omega \in \mathbb{S}^1$, determine the range of $(p,q)\in [1,\infty]^{2}$ for which there exists a constant $C=C_{\Sigma, \omega, p,q}>0$ such that the following inequality holds for all Schwartz functions $f\in \mathcal{S}(\R^2)$ uniformly in $t\in \R$
    \begin{equation}
        \label{e:radon_restr}
        \left( \mathcal{R}(\abs{E_{\Sigma}f}^{q})(\omega,t) \right)^{\frac{1}{q}} \leq C\norm{f}_{L^p(\Sigma)}.
    \end{equation}
\end{problem}

When the curve $\Sigma$ is compact, Problems \ref{prob:strip} and \ref{prob:radon} are essentially equivalent and this is the content of the first main result of our paper.

\begin{theorem}
    \label{t:compact}
    Let $\Sigma\subset \R^2$ be a compact $C^2$ curve with nonzero curvature, and let $\omega \in \mathbb{S}^1$.
    \begin{enumerate}[(a)]
        \item (Transversal case) If the condition \eqref{e:transversal} holds, then the estimates \eqref{e:resriction_strip} and \eqref{e:radon_restr} hold if and only if 
        \[\frac{1}{p}+\frac{1}{q}\leq 1 \quad \text{and}\quad  q\ge 2,\]
        that is, if and only if $(\frac{1}{p}, \frac{1}{q})$ lies inside the quadrilateral $\square OACD$ in Figure~\ref{f:pqrangeAllRelevant}.
        \item (Non-transversal case) If the condition \eqref{e:transversal} does not hold, then the estimates \eqref{e:resriction_strip} and \eqref{e:radon_restr} hold if and only if
        \[\begin{cases}
        \frac{1}{p}+\frac{2}{q}\leq 1, \; &\text{if}\; p\leq q\\
        \frac{1}{p}+\frac{2}{q}< 1, &\text{if}\; p > q,
    \end{cases}\]
    that is, if and only if the point $(\frac{1}{p},\frac{1}{q})$ lies in $\triangle OAD \setminus [D,E)$ in Figure~\ref{f:pqrangeAllRelevant}.
    \end{enumerate}
\end{theorem}



When the curve $\Sigma$ is parabola $\mathbb{P}^1$,
there is only one direction, $\omega= (0,\pm 1)$, that satisfies transversality condition \eqref{e:transversal}.
We first note the connection with estimates for Schr\"odinger operator.

For $u_0\in \mathcal{S}(\R)$, the solution of the Schr\"odinger equation
\[\begin{cases}
    2\pi i \partial_{x_2} u = \partial_{x_1}^2 u\\
    u(x_1,0)= u_0(x_1)
\end{cases}\]
is given as
\[u(x)=E_{\mathbb{P}^1}v_0(x),\] 
where $v_0\in \mathcal{S}(\R^2)$ is defined as
$v_0(\xi_1,\xi_2):=\widehat{u_0}(\xi_1)\chi(\xi_2-\xi_1^2)$ for some $\chi\in C_c^{\infty}(\R)$ with $\chi(0)=1$.

For interval $I\subset \R$, the estimates of the form 
\[\norm{u}_{L^q(\R\times I)} \lesssim \norm{u_0}_{L^p(\R)}\]
are classically considered in the theory of partial differential equations. See \cite{Rogers08, GOW22, GaoLiWang22, GaoMiaoZheng22,Schippa22, Lu23}.
It is, therefore, a natural question to consider the analogous results in which the function $u$ is replaced with $E_{\mathbb{P}^1}v_0$ and the function $u_0$ is replaced by $v_0^\vee$.

The results of that kind are the second main result of our paper.

\begin{theorem}
\label{t:parabola}
    Let $\Sigma = \mathbb{P}^1$. 
    \begin{enumerate}[(a)]
        \item (Transversal case) When $\omega = (0,\pm 1)$, the estimate \eqref{e:radon_restr} holds if and only if 
        \[\frac{1}{p}+\frac{1}{q}=1 \quad \text{and} \quad p\leq q.\]
        The estimate \eqref{e:resriction_strip} holds if and only if 
        \[\frac{1}{p}+\frac{1}{q}\leq 1, \quad \frac{1}{p}+\frac{3}{q}\ge 1, \quad\text{and}\quad (p<q \;\text{or}\; p=q=2),\]
        that is, if and only if the point $(\frac1p,\frac1q)$ belongs to $\triangle BAC\setminus[B,C)$ in Figure \ref{f:pqrangeAllRelevant}.
        \item (Non-transversal case) When $\omega\neq (0,\pm 1)$, the estimate \eqref{e:radon_restr} holds if and only if
        \[\frac{1}{p}+\frac{2}{q}=1\quad \text{and}\quad p\leq q,\]
        that is, if the point $(\frac{1}{p},\frac{1}{q})$ belongs to $[AE]$ in Figure~\ref{f:pqrangeAllRelevant}.
        The estimate \eqref{e:resriction_strip} holds if
        \[ \begin{cases}
            \frac{1}{p}+\frac{3}{q}\ge 1,\quad \frac{1}{p}+\frac{2}{q}\leq 1, \quad (p,q)\neq (4,4), \quad &\text{if}\; p\leq q\\
            \frac{1}{p}+\frac{1}{q}> \frac{1}{2},\quad  \frac{1}{p}+\frac{2}{q}< 1, &\text{if}\; p>q,
            \end{cases}
         \]
        that is, when $(\frac1p,\frac1q)$ belongs to $\triangle BAD \setminus ([B,D]\cup [D,E) )$.
        The estimate \eqref{e:resriction_strip} \underline{does not} hold in the complement of the given range except maybe in the interior of the segment $(B,D)$.
    \end{enumerate}
\end{theorem}

Observe the essential difference between this theorem and Theorem~\ref{t:compact}, where the ranges of the estimates \eqref{e:resriction_strip} and \eqref{e:radon_restr} coincide.
Let us briefly explain the reason. Restricting the space of integration to a single strip imposes a preferred scale. In the case of a compact curve, there are only finitely many wave packets at this scale, so the main obstructions arise from a single wave packet. 
In contrast, for the parabola, there are infinitely many wave packets at the same scale. 
This allows us to exploit the fact that a strip, although essentially one-dimensional, has positive width, enabling us to arrange the essential supports of wave packets into interesting geometric configurations, such as a Besicovitch-type set, within the strip.

Finally, we note that, even though it is natural to think about the given problem in terms of wave packets, the standard modern techniques (for example, those in \cite{Guth16, Wang22}) that use the $\varepsilon$-removal lemma from \cite{Tao99} to reduce the problem to a bounded region and then apply induction on scales do not seem to be well suited to this problem. 
Indeed, the full range of $(p,q)$ lies outside the critical line along which one can reduce the problem for the full parabola to a compact subset. Moreover, in the case of the full parabola, the wave packets associated with high frequencies are longer and thinner, so the induction process is also not straightforward.


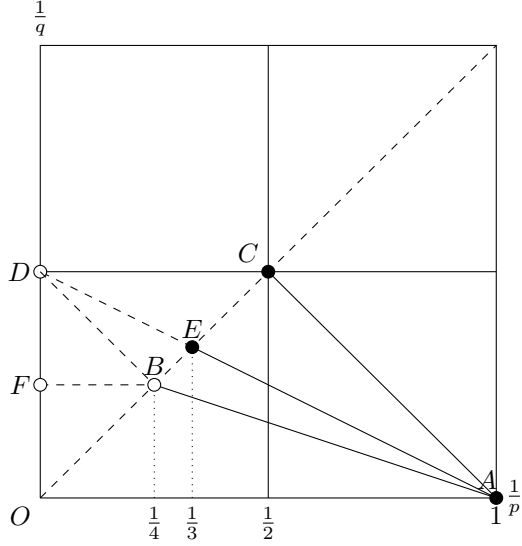
\begin{figure}
\begin{tikzpicture}[scale=6]


\draw (0,0)-- (1,0)--(1,1)--(0,1)--(0,0);
\draw [name path=half](1/2,0)--(1/2,1);
\draw [name path=diag, dashed](0,0)--(1,1);
\draw (0,1/2)--(1,1/2);

\node [right]at (1,0) {$\frac1p$};
\node [above]at (0,1) {$\frac1q$};

\draw (1,0)--(1/2,1/2);
\draw [name path=critical,dashed] (1/3,1/3)--(0,1/2);
\draw (1,0)--(1/3,1/3);

\path [name intersections={of=diag and critical, by={restriction}}];
 
\fill (1,0) circle (0.4pt);
\draw (1,0) circle (0.4pt);

\fill [white] (0,1/2) circle (0.4pt);
\draw (0,1/2) circle (0.4pt);

\fill [black] (restriction) circle (0.4pt);
\draw (restriction) circle (0.4pt);

\draw (1,0)--(1/4,1/4);
\draw [dashed] (0,1/4)--(1/4,1/4);
\draw [dashed] (0,1/2)--(1/4,1/4);
\draw [dotted] (1/4,0)--(1/4,1/4);
\draw [dotted] (1/3,0)--(1/3,1/3);

\fill [white](1/4,1/4) circle (0.4pt);
\draw (1/4,1/4) circle (0.4pt);

\node [below left] at (0,0) {$O$};

\node [below] at (1/2,0) {$\frac12$};
\node [below] at (1,0) {$1$};
\node [below] at (1/4,0) {$\frac14$};
\node [below] at (1/3,0) {$\frac13$};

\node [above] at (1-0.02,0) {$A$};
\node [above ] at (1/4,1/4) {$B$};
\node [left] at (0,1/2) {$D$};
\node [above] at (restriction) {$E$};

\fill [white](0,1/4) circle (0.4pt);
\draw (0,1/4) circle (0.4pt);
\node [left] at (0,1/4) {$F$};

\node [above left] at (1/2,1/2) {$C$};
\fill (1/2,1/2) circle (0.4pt);
\draw (1/2,1/2) circle (0.4pt);

\end{tikzpicture}
\caption{Relevant points}

\label{f:pqrangeAllRelevant}

\end{figure}

\subsection*{Structure of the paper}
In Section~\ref{s:preliminary}, we recall the known results that will be used to prove the main theorems. 
In Section~\ref{s:oscillatory}, we establish bounds for one-dimensional oscillatory integrals corresponding to $x_2\mapsto \mathcal{R}(|E_{\mathbb{P}^1}f|^q)(0,x_2)$. 
The main tool for the upper bounds is a weighted estimate for the Fourier transform proved by Pan et al. in \cite{PSS97}, while the proofs of necessity rely on Knapp-type examples and smooth approximations of homogeneous distributions, following the approach of the first author and Kova\v{c} in \cite{BK23}.
In Section~\ref{s:proof_compact}, we reduce the upper bounds of Theorem~\ref{t:compact} to the oscillatory integral bounds from Section~\ref{s:oscillatory} and extend the proofs of the necessary conditions from Section~\ref{s:oscillatory}. 
Finally, in Section~\ref{s:proof_parabola}, we prove the upper bounds in Theorem~\ref{t:parabola} by interpolating between the oscillatory integral bounds from Section~\ref{s:oscillatory} and the restriction theorem in two dimensions. To establish necessity, we use Knapp-type examples in combination with Besicovitch set estimates, following an approach similar to that of Beckner et al.~\cite{BCSS89}.

A considerable amount of the paper is devoted to intricate estimates of oscillatory integrals, but we emphasize that this approach, given the positive results obtained, appears to be essential.

\subsection*{Notation}
\label{s:notation}
If $A$, and $B$ are some expressions, then for some (possibly empty) set of indices, $P$, the expression $A\lesssim_P B$ means that there is a constant $C>0$ depending on $P$ such that $A\leq CB$. 
For a function $f:\R\to\C$ and a point $x_0\in \R$, we denote by $\tau_{x_0}$ the translation operator defined by $\tau_{x_0}f(x):=f(x-x_0)$. For a set $A$ we write $1_A$ for its characteristic function. For $r>0$ and $x_0\in \R^2$ we denote by $B(x_0,r)$ the ball of radius $r$ around $x_0$.
The unspecified constants $C>0$ in the proofs may vary from line to line.

\section{Preliminaries}
\label{s:preliminary}
We recall the known results that will be used in the proof.
\subsection{Two-dimensional restriction theorem}
We state the version of the restriction theorem that can be quickly recovered from the Corollary 1 of Section 5, Chapter IX, in Stein's book \cite{Stein93book}. We refer the interested reader also to Tao's lecture notes \cite{TaoRestNotes} for modern treatment.

\begin{theorem}[2-dimensional restriction theorem]
\label{t:restriction_2d}
\begin{enumerate}[(a)]
        \item Let $\Sigma\subset \R^2$ be a compact curve whose curvature is nowhere zero, equipped with the described measure $\mu$. Then the estimate \eqref{e:restriction_2d} holds for $p,q \in [1,\infty]$ if and only if $p^{-1}+3q^{-1}\leq 1$ and $q>4$, i.e., when $(p^{-1},q^{-1})$ lies inside the quadrilateral $\square OABF$ in the Figure~\ref{f:pqrangeAllRelevant}.
        \item When $\Sigma = \mathbb{P}^1$ with the described measure $\mu$, the estimate \eqref{e:restriction_2d} holds if and only if $p^{-1}+3q^{-1}=1$ and $q>4$, i.e., when $(p^{-1},q^{-1})$ lies on the segment $[AB)$ in Figure~\ref{f:pqrangeAllRelevant}.
    \end{enumerate}
\end{theorem}


\subsection{Knapp-type estimate}
For a function $f:\R\to \C$ and $x\in \R^2$ let
\begin{equation}
    \label{e:E_def}
    Ef(x_1,x_2):=\int_{\R}e^{2\pi i (\xi x_1+\xi^2x_2)}f(\xi)\d \xi.
\end{equation}
For $\omega_0\in \mathbb{S}^1$ and $x_0\in \R^2$ denote with
\[
R(\omega_0,x_0;\alpha,\beta)
:=
\{x\in \R^2: 
|(x-x_0)\cdot \omega_0^{\bot}|\leq \alpha, \, 
|(x-x_0)\cdot \omega_0|\leq \beta,
\}
\]
the $\alpha\times\beta$ rectangle centered at $x_0$ oriented in the direction $\omega_0$. We will usually have $\alpha \gg \beta$ so that $\omega$ will be the direction of that tiny tube. 

The following lemma is a standard result, but we provide a proof because we want to track the dependence on $\xi_0$.

\begin{lemma}
\label{l:knapp}
Let $\phi \in C_c^{\infty}(\R)$ be a real function such that $\phi(0)=1$, and let $\xi_0 \in \R$ and $a \in \R^2$. Define
\[f_{\xi_0,a}(\xi):=e^{-2\pi i \xi \cdot a}\tau_{\xi_0}\phi(\xi).\]
Then there exists a constant $c>0$ such that
\[|E f_{\xi_0,a}(x)| \gtrsim 1_{R(\omega_{\xi_0}, a;\; c(1+|\xi_0|), c/(1+|\xi_0|))}(x),\]
where $\omega_{\xi_0} = (-2\xi_0,1)/|(-2\xi_0,1)|$.
\end{lemma}
\begin{proof}
It is clear that the parameter $a \in \R^2$ only induces a translation, so we may assume without loss of generality that $a = 0$. Observe that
\[
E\tau_{\xi_0}\phi(x)
=
e^{2\pi i(\xi_0 x_1+\xi_0^2x_2)}
\int_\R e^{2\pi i(x_1+2\xi_0x_2)\xi+x_2\xi^2} \phi(\xi)\, \d \xi.
\]
Denoting 
\[\Tilde{R}(\xi_0):=\{ (x_1,x_2) : |x_1+2\xi_0 x_2|<1/10,\; |x_2|<1/10 \},\]
it is easy to see that if $(x_1,x_2) \in \Tilde{R}(\xi_0)$, then
\[|Ef_{\xi_0,a}(x)| \ge \Re\int_\R e^{2\pi i((x_1+2\xi_0x_2)\xi+x_2\xi^2)} \phi(\xi)\, \d \xi \gtrsim_{\phi} 1.\]
A simple computation
\[\begin{bmatrix}
    1 & 2\xi_0\\
    0 & 1
\end{bmatrix}\Big (\alpha \begin{bmatrix}
    -2\xi_0\\ 1
\end{bmatrix} + \beta\begin{bmatrix}
    1\\ 2\xi_0
\end{bmatrix}\Big) =\alpha \begin{bmatrix}
    0\\ 1
\end{bmatrix} + \beta \begin{bmatrix}
    1+4\xi_0^2\\
    2\xi_0,
\end{bmatrix}\]
implies that the last expression belongs to $\Tilde{R}(\xi_0)$ whenever $|\alpha| \leq c$ and $|\beta| \leq c/(1+|\xi_0|)^2$ for some small $c>0$. Therefore $\Tilde{R}(\xi_0)$ contains a rectangle of dimensions $\sim (1+|\xi_0|) \times (1+|\xi_0|)^{-1}$ in direction $\omega_{\xi_0}$ and this completes the proof.
\end{proof}

\subsection{Fourier transform estimate with a weight}
The following lemma, which first appeared in Zygmund's book \cite[p.~125]{ZygBook} in the context of Fourier series and was proved in the present form by Pan, Sampson, and Szeptycky \cite{PSS97}, is the main tool for proving estimates for the one-dimensional oscillatory integral that appears in our problem. See also paper of Xiao \cite{Xiao17} for a far-reaching generalization. We include the short proof for completeness, as we think it is not well known.
\begin{lemma}[\cite{PSS97}]
\label{l:PPS}
For any $p\in (1,2]$ the following estimate holds
    \[\int_{\R} |\widehat{f}(x)|^p|x|^{p-2}dx\lesssim_p \|f\|_{L^p(\R)}^p.
    \]
\end{lemma}

\begin{proof}

Set $d\nu(x)=|x|^{-2}\,dx$ and $\mathcal{T} f(x)=|x|\,\widehat{f}(x)$. Observe that 
\[
\|\mathcal{T}f\|_{L^2(\R, \d\nu)} 
= 
\|f\|_{L^2(\R, \d x)}
\]
and 
\[
\|\mathcal{T}f\|_{L^{1,\infty}(\R, \d\nu)}
\leq 
2 \|f\|_{L^1(\R,\d x)}.
\]
The first equality is a consequence of Plancherel's theorem. The second inequality follows from the bound
\begin{align*}
    \nu(\{x\in\R:|\mathcal{T}f(x)|>\lambda\})
    \leq 
    \nu\Big(\Big\{x\in\R:|x|>\frac{\lambda}{\|f\|_{L^1(\R, \d x)}}\Big\}\Big)
    =
    2\frac{\|f\|_{L^1(\R, \d x)}}{\lambda}
\end{align*}
for all $\lambda>0$.
By the Marcinkiewicz interpolation theorem, we conclude that 
\[
\|\mathcal{T}f\|_{L^{p}(\R, \d\nu)}
\lesssim_p
\|f\|_{L^p(\R, \d x)}
\]
for $p\in(1,2]$, 
which is the desired estimate.
\end{proof}
\section{One-dimensional oscillatory integrals}
\label{s:oscillatory}

Let $f\in \mathcal{S}(\R)$ and we define the operators $T$, $T_{\low}$ and $T_{\high}$ with:
\[
T f(x)
:=
\int_{0}^{\infty} e^{2\pi ix\xi^2}f(\xi)\,\d \xi, \quad x\in \R,
\]
\[T_{\low}f(x):=T(1_{[0,1]}f)(x) \quad \text{and} \quad T_{\high}f(x):=T(1_{[1,\infty)}f)(x).\]
They satisfy the following estimates. 

\begin{proposition}\label{p:T}
The estimate 
\begin{equation}\label{e:T}
    \|Tf\|_{L^q(\R)}
    \leq 
    C_{p,q}
    \|f\|_{L^p(\R)}
\end{equation}
holds for some positive constant $C_{p,q}<\infty$ if and only if 
\[\frac{1}{p}+\frac{2}{q}=1\quad  \text{and}\quad  p\leq q,\] 
i.e. $(\frac1p,\frac1q)$ belongs to $[A,E]$ in Figure~\ref{f:pqrangeAllRelevant}.
\end{proposition}

\begin{proposition}\label{p:Tlow}
The estimate 
\begin{equation}\label{e:Tlow}
    \|T_{\low} f\|_{L^q(\R)}
    \leq 
    C_{p,q}
    \|f\|_{L^p(\R)}
\end{equation}
holds for some positive constant $C_{p,q}<\infty$ if and only if 
\[\big(\frac{1}{p}+\frac{2}{q}=1\;\text{and}\; p\leq q\big),\quad \text{or}\quad \frac{1}{p}+\frac{2}{q}<1,\]
i.e. $(\frac1p,\frac1q)$ belongs to $\triangle OAD\backslash [D,E)$ in Figure~\ref{f:pqrangeAllRelevant}.
\end{proposition}

\begin{proposition}
    \label{p:Thigh}
    The estimate
    \begin{equation}
        \|T_\high f\|_{L^q(\R)}\leq C_{p,q}\|f\|_{L^p(\R)}
    \end{equation}
    holds if
    \[\big(\frac{1}{p}+\frac{2}{q}=1\; \text{and}\; p\leq q\big),\quad \text{or}\quad \big( \frac{1}{p}+\frac{2}{q}>1\; \text{and}\; \frac{1}{p}+\frac{1}{q}< 1,\; \text{and}\; q>2  \big),\]
    i.e. $(\frac1p,\frac1q)$ belongs to $\triangle ACD\setminus ([A,C]\cup [C,D] \cup [D,E))$ in Figure~\ref{f:pqrangeAllRelevant}.
\end{proposition}
\textit{Note.} We did not pursue the proof of the necessity of the conditions in Proposition~\ref{p:Thigh} because we will not need them in the proofs of the main theorems. However, one can immediately see that the estimate cannot hold for any other $(\frac{1}{p},\frac{1}{q})$ inside the triangle $\triangle OAD$, since otherwise the operator $T$ would be bounded at that point.

\begin{proof}[Proof of Proposition~\ref{p:T}]
We begin by proving the sufficiency of the statement. 
Since $T$ is trivially bounded from $L^1(\R)$ to $ L^\infty(\R)$, by complex interpolation, it is enough to verify the endpoint $L^3(\R)\to L^3(\R)$ estimate.
Taking the dual form and using change of variables, this is equivalent to 
\[
\Big|\int_{0}^{\infty}f(x^{\frac{1}{2}})x^{-\frac{1}{2}}\widehat{g}(x)dx \Big|
\lesssim
\norm{f}_{L^3(\R)}
\norm{g}_{L^{3/2}(\R)}. 
\]
However, this follows from H\"older's inequality and Lemma~\ref{l:PPS} with $p=\frac32$:
\begin{align*}
\Big|\int_{0}^{\infty}f(x^{\frac{1}{2}})x^{-\frac{1}{2}}\widehat{g}(x)dx \Big|
&\leq
\Big(\int_0^\infty |f(x^\frac12)|^3x^{-\frac12},\d x\Big)^\frac13
\Big(\int_0^\infty |\widehat{g}(x)|^\frac32 x^{\frac32-2}\,\d x\Big)^\frac23
\leq
\|f\|_{L^3(\R)} \|g\|_{L^{\frac32}(\R)}.
\end{align*}

\bigskip

It requires more work to show the necessity of the condition as we are dealing with lower bounds of oscillatory integrals. 
First observe that necessity of the condition $\frac{1}{p}+\frac{2}{q}=1$ follows from dilation symmetry
\[T[f(\lambda \cdot)](x)=\lambda^{-1}Tf(\lambda^{-2}x).\]

We now proceed to prove the necessity of the condition $p\leq q$.
From the second sufficient condition in Proposition~\ref{p:Tlow}, which we will prove shortly,
it follows that a simple bump function cannot be used here, as its $L^p$ norms are comparable.
Instead, we must test the operator on a function whose $L^p$ norm is much smaller than its $L^{p'}$ norm for all $p' > p$. This is morally true for the function $x \mapsto |x|^{-1/p}$, but to formalize this intuition we use the approximation of homogeneous distribution from \cite[\S 4]{SW71Book}, further developed in \cite{BK23} by Kova\v{c} and the first author.

For $\varepsilon\in(0,1)$ let 
\[
f_\eps(\xi)
:=
\int_{\eps}^{1/\varepsilon} e^{-\pi \frac{\xi^2}{t^2}}t^{-\frac{1}{p}-1}\, \d t.
\]
Lemma 6 in \cite{BK23} implies that $\|f_\eps\|_{L^p([0,\infty))}\lesssim_p (\log\eps^{-1})^{\frac{1}{p}}$. However, we provide a short self-contained proof. Using change of variables, we have
\[\|f_\eps\|_{L^p([0,\infty))}^p \sim_p \int_{0}^{\infty} \xi^{-1}\Bigl|\int_{\pi\eps^2\xi^2}^{\pi(1/\eps^2)\xi^2} e^{- u}u^{\frac{1}{2p}-1}\, \d u \Bigr|^pd\xi.\]
For $\xi \in [\varepsilon, \varepsilon^{-1}]$, we use the fact that the inner integral is $O(1)$, as it is a truncation of the Gamma function, so integration in $\xi$ gives $O(\log \varepsilon^{-1})$.
When $\xi \leq \varepsilon$, we use $e^{- u} \leq 1$ to see that the inner integral is $O(\varepsilon^{-1} \xi)$, which gives an $O(1)$ contribution upon integration in $\xi$.
Finally, for $\xi \ge \varepsilon^{-1}$, we use the estimate $u^{\frac{1}{2p}-1} \leq 1$ to conclude that the inner integral is $O(e^{-\varepsilon^2 \xi^2})$, which again gives an $O(1)$ contribution upon integration in $\xi$.

On the other hand, by appropriate change of variables in $t$ and $\xi$ we get
\[T f_\eps(x)=x^{-\frac1q}
\int_{0}^{\infty}\int_{\eps \sqrt{x}}^{(1/\varepsilon)\sqrt{x}} 
e^{-\pi(\frac{1}{t^2} - 2 i)\xi^2} t^{-\frac{1}{p}-1}\,\d t\d\xi,\]
under the condition $\frac1p+\frac2q=1$. Invoking the fact that 
\begin{equation}\label{e:gauss int}
\int_{0}^{\infty}e^{-\alpha u^2}du
=
\frac{1}{2}\sqrt{\frac{\pi}{\alpha}}
\end{equation}
whenever $\Re \alpha >0$, we obtain
\begin{align*}
T f_\varepsilon(x)
&=
\frac{x^{-\frac{1}{q}}}{2}
\int_{\eps \sqrt{x}}^{(1/\eps)\sqrt{x}} 
\frac{1}{\sqrt{\frac{1}{t^2}-2i}} 
t^{-\frac{1}{p}-1}\,\d t
=
\frac{x^{-\frac{1}{q}}}{2}
\int_{\eps \sqrt{x}}^{(1/\eps)\sqrt{x}}
t^{-\frac{1}{p}} (4t^4+1)^{-\frac{1}{4}} e^{-\frac{i}{2}\arg(\frac{1}{t^2}-2i)}\,\d t.
\end{align*}
Now, note that we have $\Re (e^{-\frac{i}{2}\arg(\frac{1}{t^2}-2 i)}) \ge \frac{\sqrt{2}}{2}$ since $\frac12\arg(\frac{1}{t^2}-2 i)\in [-\frac{\pi}{4},0]$ (strictly away from $-\frac\pi2$).
It follows that 
\begin{align*}
    |Tf_\eps(x)|
    \geq \Re (Tf_\eps(x))=
    \frac{\sqrt{2}}{4} x^{-\frac{1}{q}}
    \int_{\eps \sqrt{x}}^{(1/\eps) \sqrt{x}}
    t^{-\frac{1}{p}} (4t^4+1)^{-\frac{1}{4}}\, \d t.
\end{align*}
In particular, for $x\in [1,1/(4\varepsilon^2)]$, since the integrand is positive, it follows that 
\[
\int_{\eps \sqrt{x}}^{(1/\eps)\sqrt{x}}
t^{-\frac{1}{p}} (4t^4+1)^{-\frac{1}{4}}\, \d t 
\ge 
5^{-\frac{1}{4}}\int_{\frac{1}{4}}^{1} t^{-\frac{1}{p}}\, \d t
\gtrsim_p 
1
\]
Finally, 
\[
\|T f_\varepsilon\|_{L^q(\R)}
\geq
\|T f_\varepsilon\|_{L^q([1,1/(4\varepsilon^2)])}
\gtrsim
(\log \varepsilon^{-1})^\frac1q.
\]
Letting $\varepsilon\to0$, we conclude that $p\leq q$ is necessary.
\end{proof}

\begin{proof}[Proof of Proposition~\ref{p:Tlow}]

We begin by proving the sufficiency of the statement. The first condition is the same as in Proposition \ref{p:T}, so the statement follows from Proposition \ref{p:T}. We turn to the proof of the sufficiency when 
\begin{equation}
    \label{e:condition<AD}
    \frac{1}{p}+\frac{2}{q}<1.
\end{equation}
By duality, the statement is equivalent to
\[\Big| \int_{0}^1f(x^{\frac{1}{2}})\widehat{g}(x)x^{-\frac{1}{2}}\d x\Big| \lesssim \norm{f}_{L^p([0,1])}\norm{g}_{L^{q'}(\R)}.\]
Using H\"older's inequality followed by the Hausdorff--Young inequality, we obtain
\begin{align*}
    \Big|\int_{0}^1f(x^{\frac{1}{2}})\widehat{g}(x)x^{-\frac{1}{2}}\d x\Big|
    &\leq \norm{f(x^{\frac{1}{2}})x^{-\frac{1}{2}}}_{L_x^{q'}([0,1])} \norm{g}_{L^{q'}(\R)}.
\end{align*}
Observing that condition \eqref{e:condition<AD} implies $r:=\frac{p}{q'}>1$, we can apply another change of variables and H\"older's inequality to conclude
\begin{align*}
    \|f(x^{\frac{1}{2}})x^{-\frac{1}{2}}\|_{L_x^{q'}([0,1])} 
    \leq \|f\|_{L^p([0,1])} \Big(\int_{0}^{1}x^{(-q'+1)r'}\d x \Big)^{\frac{1}{r'q'}} \lesssim_{p,q} \|f\|_{L^p([0,1])},
\end{align*}
where the last inequality holds because condition \eqref{e:condition<AD} is equivalent to $(-q'+1)r'>-1$, ensuring that $x\mapsto x^{(-q'+1)r'}$ is a locally integrable on $[0,1]$. 

\bigskip
We turn to the proof of necessity. Again, the condition $\frac{1}{p}+\frac{2}{q}\leq 1$ follows from the dilation symmetry for small dilation. 

The proof that $p\leq q$ is necessary when $\frac{1}{p}+\frac{2}{q}=1$ follows the one for Proposition~\ref{p:T} up to the point where the identity \eqref{e:gauss int} is invoked. Since the operator $T_\low$ is localized, the identity \eqref{e:gauss int} is no longer available, and an additional argument is required to proceed. 
To produce the required lower bound for the oscillatory integral, we employ the method of steepest descent. 

\begin{figure}
\begin{tikzpicture}[scale=3]
\tikzset{->-/.style={decoration={
  markings,
  mark=at position .5 with {\arrow{stealth}}},postaction={decorate}}}

\begin{scope}
  \path[name path=C] (1.5,0) arc[start angle=0, end angle=90, radius=1.5];

  \path[name path=line] (0,0) -- (1,2);

  \path[name intersections={of=C and line, by=X}];
  
  \draw (0,0) -- (1.5,0) arc[start angle=0, end angle=90, radius=1.5];
  
  \pgfmathanglebetweenpoints{\pgfpoint{0}{0}}{\pgfpointanchor{X}{center}}
  \let\endangle\pgfmathresult

  \draw[->] (-0.1,0)--(2,0) node [right] {$\Re$};
  \draw[->] (0,-0.3)--(0,1.7) node [above] {$\Im$};


  \draw [] (0.2,0) arc[start angle=0, end angle=\endangle, radius=0.2];
  
  \draw [red,thick,->-](X) -- (0,0);
  \draw [red,thick,->-] (0,0)--(1.5,0);
  \draw [red,thick,->-] (1.5,0) arc[start angle=0, end angle=\endangle, radius=1.5];

\begin{scope}[shift={(0.3,-0.3)}]
  \draw [blue,->-] (0,0)--(1.5,0);


  \coordinate (A) at (0,0);
  \coordinate (B) at (1.5,0);
  
\end{scope}
  
  \fill[] (A) circle (0.5pt);
  \fill[] (B) circle (0.5pt);
  \draw[blue,->-] (0,0)--(A);
  \draw[blue,->-] (1.5,0)--(B);

  \node [right]at (0.16,-0.15) {$\rho_1(z_0)$};
  \node [right]at (1.5+0.16,-0.15) {$\rho_2(z_0)$};

  \node [] at (0.3,0.15) {$-\frac{\vartheta(t)}{2}$};

  \fill[] (X) circle (0.5pt);
  \fill[] (1.5,0) circle (0.5pt);
  \fill[] (0,0) circle (0.5pt);

  \node[above right] at (X) {$Q$};
  \node [above right] at (1.5,0) {$P$};
  \node [above right] at (1.5/2,0) {$\gamma_1$};
  \node [below right,shift={(0.3,0)}] at (1.5/2,-0.3) {$\gamma_1-z_0$};
  \node [right] at (1.2,1) {$\gamma_2$};
  \node [left] at (0.5,1) {$\gamma_3$};
  \node [below] at (1.5,0) {$\sqrt{x}$};

  \node [below left] at (0,0) {$O$};

  \end{scope}

\end{tikzpicture}
\caption{The contour for the integral in $\C$.}
\label{f:contour}
\end{figure}
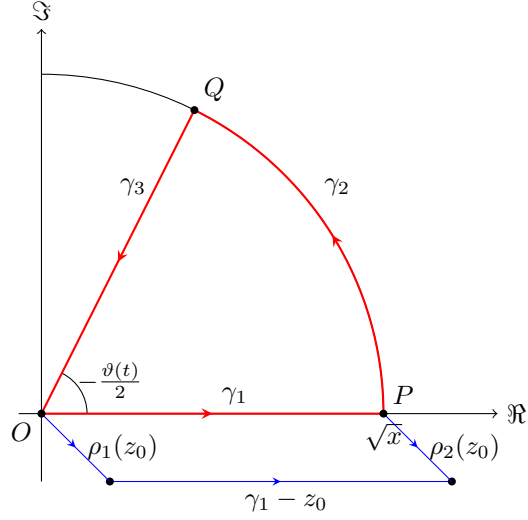

Let $\psi(t)=t^{-2}-2i$ and $\vartheta(t)=\arg\psi(t)$. 
Using the same change of variables as in the proof of Proposition~\ref{p:T}, we arrive to
\begin{equation}
    \label{e:T_low_expanded}
    T_\low f_{\eps}(x) =x^{-\frac{1}{q}}\int_{\eps\sqrt{x}}^{(1/\eps)\sqrt{x}} t^{-\frac{1}{p}-1}\int_{0}^{\sqrt{x}}e^{-\pi \psi(t)\xi^2}\,\d\xi \d t.
\end{equation}

Consider the path integral along the curve $OPQ$ in Figure~\ref{f:contour}. By Cauchy's integral theorem, 
we have
\[
\int_{0}^{\sqrt{x}} e^{-\pi\psi(t)\xi^2} \d \xi
=
\int_{-\gamma_3} e^{-\pi\psi(t) z^2} \d z
-
\int_{\gamma_2} e^{-\pi\psi(t) z^2} \d z,
\]
where $\gamma_2$ denotes the arc $PQ$ (with angle $-\vartheta(t)/2$) and $-\gamma_3$ denotes the line from $O$ to $Q$.

Let us first deal with the integral along $-\gamma_3$. After a suitable change of variables, we obtain 
\[
\int_{-\gamma_3} e^{-\pi\psi(t) z^2} \d z
=
\frac{e^{-i\vartheta(t)/2}}{|\psi(t)|^\frac12} 
\int_0^{\sqrt{x|\psi(t)|}} e^{-\pi r^2} \d r.
\]
Note the elementary inequality
\[\int_{0}^{1}e^{-y^2}dy\ge \int_{0}^{1}(1-y^2) \d y= \frac{2}{3},\]
together with the fact that $|\psi(t)|\ge 1$ for any $t>0$. 
Using these observations and following the same reasoning as in the proof of Proposition~\ref{p:T}, we deduce that for any $x\ge 1$ the following estimate holds:
\[\Re\Big(\int_{-\gamma_3} e^{-\pi\psi(t) z^2} \d z \Big)=
\Re \Big(\frac{e^{-i\vartheta(t)/2}}{|\psi(t)|^\frac12}\int_0^{\sqrt{x|\psi(t)|}} e^{-\pi r^2} \d r \Big)
\gtrsim \frac{1}{|\psi(t)|^{\frac{1}{2}}}.
\]

Next, consider the integral along the arc $\gamma_2$. Using parametrization $\theta\mapsto \sqrt{x}e^{i\theta}$ and triangle inequality, we obtain
\begin{equation}
\label{e:int_gamma2_temp_0}
    \begin{split}
        \Big|\int_{\gamma_2} e^{-\pi\psi(t) z^2} \d z\Big| &= \Big|\sqrt{x}\int_{0}^{-\frac{\vartheta(t)}{2}}e^{-\pi \psi(t)x(\cos(2\theta)+i\sin(2\theta))} e^{i\theta}\,\d \theta\Big|\\
        &\leq \sqrt{x} \int_0^{-\frac{\vartheta(t)}{2}} e^{-\pi x(t^{-2}\cos(2\theta)+2\sin(2\theta))}\, \d \theta ,
    \end{split}
\end{equation}
where $-\vartheta(t)\in (0,\frac\pi2)$. 
Using elementary inequality $-\vartheta(t)=\arctan(2t^2)\leq 2t^2$, we obtain:
\[
t^{-2}\cos(2\theta)+2\sin(2\theta) \ge t^{-2}\cos(2t^2)\ge t^{-2}/2,\quad t\in [0,1/2].
\]
Therefore, using \eqref{e:int_gamma2_temp_0} and elementary inequality $e^{-u}\leq u^{-1},$ the following inequality holds:
\begin{equation}
\label{e:int_gamma2_temp_1}
     \Big|\int_{\gamma_2} e^{-\pi\psi(t) z^2} \d z\Big|\leq \sqrt{x}\int_{0}^{t^2}e^{-\pi \frac{x}{2t^2}} \d\theta\leq \sqrt{x}t^2e^{-\pi\frac{x}{2t^2}}\leq t^4x^{-\frac{1}{2}},\quad t\in\Big[0,\frac{1}{2} \Big] 
\end{equation}
On the other hand, when $t\ge \frac{1}{2}$, we can use \eqref{e:int_gamma2_temp_0} and the elementary estimate $\sin u\ge \frac{2}{\pi}u$, $u\in [0,\frac{\pi}{2}]$ to conclude:
\begin{equation}
    \label{e:int_gamma2_temp_2}
    \Big|\int_{\gamma_2} e^{-\pi\psi(t) z^2} \d z\Big| \leq \sqrt{x}\int_{0}^{\frac{\pi}{4}} e^{-4 x \theta}\, \d\theta \leq x^{-\frac{1}{2}}
\end{equation}
Combining estimates \eqref{e:int_gamma2_temp_1} and \eqref{e:int_gamma2_temp_2}, and using $|\psi(t)|^{-1}\sim \min\{t^2,1\}$ it follows that
\begin{equation}
    \label{e:int_gamma2}
    \Big|\int_{\gamma_2} e^{-\pi\psi(t) z^2} \d z\Big|
\leq x^{-\frac12}(t^41_{[0,1/2]}(t)+1_{[1,\infty)})\lesssim x^{-\frac12}|\psi(t)|^{-\frac{1}{2}}.
\end{equation}
Using triangle inequality we conclude that
\begin{align*}
            \Re \Big(\int_{0}^{\sqrt{x}} e^{-\pi\psi(t)\xi^2} \d \xi \Big) &\ge \Re \Big( \int_{-\gamma_3} e^{-\pi\psi(t)\xi^2}\, \d \xi\Big) - \Big|\int_{-\gamma_2} e^{-\pi\psi(t)\xi^2}\, \d \xi \Big|
    \ge (c_1-c_2x^{-\frac{1}{2}})|\psi(t)|^{-\frac{1}{2}}.
\end{align*}
Therefore, for some $c_*\ge 1$ sufficiently large depending on $c_1,c_2$ and $x\ge c_*$, the following inequality holds:
\begin{equation}
    \label{e:int_gamma1}
    \Re \Big(\int_{0}^{\sqrt{x}} e^{-\pi\psi(t)\xi^2} \d \xi \Big) \gtrsim |\psi(t)|^{-\frac{1}{2}}.
\end{equation}
Finally, inserting \eqref{e:int_gamma1} into \eqref{e:T_low_expanded}, for $x\in [c_*,\frac{\eps^2}{4}]$, the following holds
\begin{align*}
|x^\frac1q T_{\low} f(x)|
&=\Big|\int_{\varepsilon\sqrt{x}}^{(1/\varepsilon)\sqrt{x}} 
t^{-\frac1p-1}
\int_{0}^{\sqrt{x}} e^{-\pi\psi(t) \xi^2} \d \xi\d t\Big|
 \geq
\Re\Big(\int_{\varepsilon\sqrt{x}}^{(1/\varepsilon)\sqrt{x}}
t^{-\frac1p-1} 
\int_{0}^{\sqrt{x}} e^{-\pi\psi(t) \xi^2} \d \xi  \,\d t\Big)\\
&\gtrsim \int_{\varepsilon\sqrt{x}}^{(1/\varepsilon)\sqrt{x}} t^{-\frac{1}{p}}(1+4t^4)^{-\frac{1}{4}} \d t
\geq \int_{1/4}^{1} t^{-\frac{1}{p}}(1+4t^4)^{-\frac{1}{4}}\d t\gtrsim 1.
\end{align*}
Therefore,
\[
|T_\low f(x)| \gtrsim x^{-\frac1q}1_{[c_*, \frac{\eps^{-2}}{4}]},
\]
which yields, for $\eps>0$ small enough, so that $\eps^{-2}>4 c_*$,
\[
\|T_{\low} f\|_{L^q(\R)}
\geq 
\|T_{\low} f\|_{L^q([c_*,1/(4\varepsilon^2)])}
\gtrsim 
|2\log(\varepsilon^{-1}) -\log (4c_*)|^\frac1q.
\]
Letting $\eps\to 0$, we conclude the proof.
\end{proof}

\begin{proof}[Proof of Proposition \ref{p:Thigh}]
One can follow the proof for Proposition~\ref{p:Tlow} in the $\frac{1}{p}+\frac{2}{q}<1$ case verbatim. The condition $\frac1p+\frac2q>1$ is used in the final step
\[
\|f(x^\frac12)x^{-\frac12}\|_{L^{q'}_x}
\leq 
\|f\|_{L^p(\R)}\Big( \int_1^\infty x^{(-q'+1)r'} \,\d x\Big)\lesssim\|f\|_{L^p(\R)}
\]
to ensure that the integral is finite.  
\end{proof}


\section{Proof of Theorem~\ref{t:compact}}
\label{s:proof_compact}
First observe that using change of variables, one has
\begin{align*}
  \|E_{\Sigma}f\|_{L^q(\{|x\cdot \omega-t|\leq 1/2\})}^q 
  = \int_{-1/2}^{1/2} \mathcal{R}(|E_{\Sigma}f|^q)(\omega,t)\d t
  \leq \sup_{t\in [-1/2,1/2]} \mathcal{R}(|E_{\Sigma}f|^q)(\omega,t).  
\end{align*}
Therefore, the estimate \eqref{e:resriction_strip} follows from the estimate \eqref{e:radon_restr}.

\subsection{Initial reductions}

Since the curve is compact, it can be locally parametrized as the graph of a function $h:I\to \R$ so that $\xi_2=h(\xi_1)$ (or vice a versa). We can choose the parametrization so that $|h'(\xi_1)|\sim 1$ for all $\xi_1\in I$ by selecting the appropriate parametrization variable. 
Then we also have $|h''(\xi_1)|\gtrsim 1$ due to the fact that the curvature $\kappa$ satisfies $\kappa(\xi_1,h(\xi_1)) = \frac{|h''(\xi_1)|}{(1+h'(\xi_1)^2)} \sim |h''(\xi_1)|$.

Therefore, the operator $E_{\Sigma}$ can be written as a finite sum of operators in which $\Sigma$ is the graph of the function $h$ over $I$, namely,
\[
E_\Sigma f(x_1,x_2)=\int_{I} e^{2\pi i (u,h(u))\cdot (x_1,x_2)}f(u,h(u))J_h(u)\d u,\]
where $I$ is an interval and $J_h(u):=(1+h'(u)^2)^{1/2}$ is a volume factor.

Observe that $J_h(u)\sim 1$ for all $u\in I$ for our choice of parametrization. 
Denoting 
\[f_1(u):=f(u,h(u))J_h(u),\] 
we can see that
\[\|f_1\|_{L^p(\R)}^p = \int_{I}|f(u,h(u))J_h(u)|^p \d u\sim \int_{I} |f(u,h(u))|^pJ_h(u)du = \|f\|_{L^p(\Sigma)}^p.\]
Therefore, in order to prove \eqref{e:radon_restr}, it suffices to prove that
\begin{equation}
    \label{e:h_sufficient}
    (\mathcal{R}(|E_{h}f|^q)(\omega,t))^{\frac{1}{q}}\lesssim \|f\|_{L^p(\R)},
\end{equation}
where $E_h$ is defined by:
\begin{equation}
    \label{e:E_h_def}
    E_hf(x):=\int_{I} e^{2\pi i (u,h(u))\cdot (x_1,x_2)}f(u)\d u
\end{equation}

We continue with reductions.
For any $u_0\in I$, we can write 
\[h(u)=h(u_0)+h'(u_0)(u-u_0)+r_2(u-u_0),\] 
so that $r_2:I_0\to \R$ satisfies $r_2(0)=r_2'(0)=0$. Using
\begin{align*}
    E_hf(x_1,x_2)&=e^{2\pi i (u_0x_1+h(u_0)x_2)}\int_{I}e^{2\pi i((u-u_0)(x_1+h'(u_0)x_2)+ r_2(u-u_0)x_2 )}f(u)\d u\\
    &=e^{2\pi i (u_0x_1+h(u_0)x_2)}E_{r_2}\tau_{-u_0}f(x_1+h'(u_0)x_2,x_2),  
\end{align*}
we conclude that
\begin{align}
\label{e:E_h_change_of_var}
    |E_hf(x_1,x_2)| = |E_{r_2}\tau_{-u_0}f(x_1+h'(u_0)x_2,x_2)|.
\end{align}
Denoting
\[\omega(h'(u_0))=\begin{bmatrix}
    1 & 0\\
    -h'(u_0) & 1
\end{bmatrix} \begin{bmatrix}
    \omega_1\\ 
    \omega_2 
\end{bmatrix}, \]
and using change of variables together with
\[\mathcal{R}g(\omega,t)=\lim_{\delta \to 0}\frac{1}{\delta}\int_{\{|x\cdot \omega - t|\leq \delta\}}g(x)\, \d x,\]
one can check that
\begin{equation}
    \label{e:radon_change_var}
    \mathcal{R}(|E_{h}f|^q)(\omega, t) = |\omega(h'(u_0))|^{-1} \mathcal{R}(|E_{r_2}\tau_{-u_0}f|^q)\big(\frac{\omega(h'(u_0))}{|\omega(h'(u_0))|},\frac{t}{|\omega(h'(u_0))|}\big).
\end{equation}

Since $\abs{\omega(h'(u_0))}\sim 1$, and since the right hand side of \eqref{e:h_sufficient} is translation invariant, we conclude that it is sufficient to prove the estimate \eqref{e:h_sufficient} under the additional assumption that $I$ is an interval around $0$ and $h(0)=h'(0)=0$. 

\subsection{Transversal case}
This calculation is the same as in \cite{BN21}, but since we work with general curves, we first needed to decompose it, and we include a short proof for completeness.
\begin{proof}
    Recall the assumptions that $0\in I$ and $h(0)=h'(0)=0$. 
    
    Operator $E_{h}$ is trivially bounded from $L^1(\Sigma)$ to $L^{\infty}(\R^2)$. Therefore, by complex interpolation and H\"older's inequality, it suffices to prove the estimate \eqref{e:h_sufficient} for $p=q=2$. 
    
    Since $0\in I$ and $h'(0)=0$, we must have $\omega\neq (1,0)$. 
    Hence, we can parametrize the set $\{x\cdot \omega =t\}$ with $x_2=\alpha x_1+\beta$ for some $\alpha, \beta\in \R$ and observe that $\omega \parallel (\alpha,-1)$.
    
    The condition is equivalent to the statement that $\omega$ is not perpendicular to any normal vector $n(u)=(-h'(u),1)$ and since $I$ is compact, we must have
    \begin{equation*}
        \abs{1+\alpha h'(u)}=\abs{(\alpha, -1)\cdot (-h'(u),1)}\sim_{\omega} \abs{\omega \cdot n(u)} \gtrsim_{\omega} 1,\quad u\in I.
    \end{equation*}
    This implies that the function $\psi(u):=u+\alpha h(u)$ satisfies 
    \begin{equation}
        \label{e:separation}
        \abs{\psi'(u)}\gtrsim_{\omega} 1, \quad u\in I.
    \end{equation}
    Substitution $v=\psi(u)$ gives
    \begin{align*}
       \mathcal{R}(\abs{E_h}^2)(\omega,t)
       &\sim_{\alpha} \int_{\R}\Big|\int_{I} e^{2\pi i (u,h(u))(x_1,\alpha x_1 +\beta)} f(u) \d u\Big|^2\d x_1
        = \int_{\R}\Big|\int_{I} e^{2\pi i x_1( u+ \alpha h(u)) } e^{2\pi i\beta h(u)}f(u) \d u\Big|^2\d x_1\\
       &= \int_{\R}\Big|\int_{\R} e^{2\pi i x_1 v } F(v)dv\Big|^2\d x_1
    \end{align*}
    where
    \[F(v):=e^{2\pi i \beta h(\psi^{-1}(v))}(f\cdot 1_I)(\psi^{-1}(v))J_{\psi^{-1}}(v).\]
    By Plancherel's theorem, the last expression is, because of \eqref{e:separation}, bounded with
    \begin{align*}
        \norm{F}_{L^2(\R)}^2=\int_{\R} \abs{(f\cdot 1_I)(\psi^{-1}(v))J_{\psi^{-1}}(v)}^2dv =\int_{I} \abs{f(u)}^2\frac{1}{\abs{\psi'(u)}} du\lesssim \norm{f}_{L^2(\R)}^2.
    \end{align*}
    Therefore, the statement is proved.
\end{proof}

\subsection{Non-transversal case}
\begin{proof}
    Since the curvature is nonzero, the function $h''$ has constant sign on $I$. This implies that the function $u\mapsto (1,h'(u_0))\cdot \omega^{\bot}$ is strictly monotone and therefore it can have at most $1$ zero in $I$. After a change of variables \eqref{e:radon_change_var} at point $u_0$, the problem is reduced to proving
    \[(\mathcal{R}\abs{E_{h}f}^{q}((1,0),t))^{\frac{1}{q}} \lesssim \norm{f}_{L^p(\R)},
    \]
    where the function $h$ satisfies $h(0)=h'(0)=0$.
    Observe that 
    \begin{equation}
        \label{e:E_h_reduced}
        (\mathcal{R}(\abs{E_{h}f}^{q})((1,0),t))^{\frac{1}{q}} = \Big(\int_{\R} \big|\int_{I} e^{2\pi i h(u)x_2} e^{2\pi itu}f(u)\d u\big|^q \d x_2\Big)^{\frac{1}{q}}.
    \end{equation}
    Dividing the domain of integration in the inner integral into positive and negative part, we can assume that $I\subset [0,\infty)$. Since $h''\gtrsim 1$ and $h'(0)=0$, it follows that $h$ is strictly increasing on $I$, so we can make a substitution $h(u)=y^2$. Denoting $F(u):=e^{2\pi i t u}f(u)$, the inner integral becomes
    \begin{equation*}
        \int_{I'}e^{2\pi i y^2} F(h^{-1}(y^2)) \frac{2y}{h'(h^{-1}(y^2))}\d u.
    \end{equation*}
    Therefore, by Proposition \ref{p:Tlow}, the expression \eqref{e:E_h_reduced} is bounded by
    \begin{equation}
        \label{e:compact_proof_penultimate}
        \Big\|F(h^{-1}(y^2)) \frac{2y}{h'(h^{-1}(y^2))}\Big\|_{L_y^p(\R)} = \Big(\int_{I} \abs{F(u)}^{p} \Big( \frac{h(u)}{h'(u)^2} \Big)^{\frac{p-1}{2}}\d u \Big)^{\frac{1}{p}}
    \end{equation}
    Finally, observe that a generalized Lagrange's intermediate value theorem implies that for any $u \ge 0$ there exists $u_1\in (0,u)$ so that
    \[ \frac{h(u)}{h'(u)^2} = \frac{h(u)-h(0)}{h'(u)^2-h'(0)^2}= \frac{h'(u_1)}{2h'(u_1)h''(u_1)}=\frac{1}{2h''(u_1)}\lesssim 1. \]
    Therefore,
    \[\eqref{e:compact_proof_penultimate}\lesssim \norm{F}_{L^p(\R)} = \norm{f}_{L^p(\R)},
    \]
    which completes the proof.
\end{proof}

\subsection{Necessity of the conditions}

\subsubsection{Necessity of $1/p+1/q\leq 1$ in the transversal case}\label{sss:antidiagonal}
\begin{proof}
    We prove that
    \begin{equation}\label{e:condition anti-diag}
    \frac{1}{p} + \frac{1}{q} \leq 1
    \end{equation}
    is a necessary condition for the estimate \eqref{e:resriction_strip} to hold in the transversal case.
    The argument is based on a standard Knapp-type example.
    
    Let $\xi_0\in \Sigma$ be any point, let $\phi\in C_c^{\infty}(\R^2)$ be a bump function such that $\phi(0)=1$ and $\delta >0$ a small real number. Define 
    \[f_\delta(\xi):=\phi((\xi -\xi_0) /\delta).
    \] 
    Then for $\delta>0$ small enoguh
    \[\|f\|_{L^p(\Sigma)}\sim \delta^{\frac{1}{p}.}\]
    On the other hand, using the same calculations as in Lemma~\ref{l:knapp}, we have
    \begin{equation*}
    |E_{\Sigma}f_\delta(x)| \gtrsim \delta \cdot 1_{R(n(\xi_0),0,c\delta^{-2},c\delta^{-1})}(x).
    \end{equation*}
    When the transversality condition \eqref{e:transversal} holds, we obtain the following lower bound:
    \begin{equation}\label{e:strip_cap_knapp_measure}
    |\{x : |x \cdot \omega| \leq 1/2\} \cap R(n(\xi),0,c\delta^{-2},c\delta^{-1})| \gtrsim \delta^{-1}.
    \end{equation}
    Therefore, if \eqref{e:resriction_strip} holds, we must have
    \[
    \delta^{1-\frac{1}{q}}\lesssim \|E_\Sigma f_\delta\|_{L^q(\{|x\cdot \omega|\leq 1/2\})} \lesssim \norm{f_\delta}_{L^p(\Sigma)}\lesssim \delta^{\frac{1}{p}}.
    \]
    Letting $\delta \to 0$, yields the condition \eqref{e:condition anti-diag}.
\end{proof}

\subsubsection{Necessity of $q\ge 2$ in the transversal case}
\label{sss:q>=2}
\begin{proof}
    By the parabolic rescaling \eqref{e:E_h_change_of_var}, the problem can be reduced to $h(x)=x^2$ case, defined on a small neighborhood $I$ of $0$.
    Let $\phi\in C_c^{\infty}(\R)$ be a nonzero function with support in $[-0,1,0.1]$ and let $J\in \N$ be a large number. There are $cJ$ indices $j$ such that the supports of translates $\phi(J\xi-j)$ are contained in $I$, for some $c>0$.

    Let $f=\sum_{j=0}^{cJ}\epsilon_jf_j$
    where $\epsilon_j\in \{\pm 1\}$ will be chosen shortly, and
    \[f_j(\xi):= \phi(J\xi-j)\] 
    Since $f_j$ have mutually disjoint supports,
    \[\|f\|_{L^p(\R)} =\Bigl(\sum_{j=1}^{cJ}\|f_j\|_{L^p(\R)}^p\Bigr)^{\frac{1}{p}}\sim_p 1.\]
    On the other hand, using Khinchine inequality,
    \[\mathbb{E}_{(\epsilon_j)}\Bigl\|\sum_{j=1}^{cJ}\epsilon_jE_hf_j \Bigr\|_{L^q(\{|x\cdot \omega|\leq 1/2\})} \gtrsim_q \Bigl\|\Big(\sum_{j=1}^{cJ}|E_hf_j|^2 \Bigl)^{\frac{1}{2}}\Bigr\|_{L^q(\{|x\cdot \omega|\leq 1/2\})}.\]
    Therefore, there exists some choice of $\epsilon_j$ such that
    \[\|E_{h}f\|_{L^q(\{|x\cdot \omega|\leq 1/2\})}\gtrsim_q \Bigl\|\Big(\sum_{j=1}^{cJ}|Ef_j|^2 \Bigl)^{\frac{1}{2}}\Bigr\|_{L^q(\{|x\cdot \omega|\leq 1/2\})}.\]
    Observe the following identity. Denoting $g_\lambda(\xi):=g(\lambda \xi)$, one has 
    \[E_hg_\lambda(x_1,x_2)=\lambda^{-1}E_hg(x_1/\lambda,x_2/\lambda^2).
    \]
    
    Therefore, using that equality and Lemma \ref{l:knapp} with $|\xi_0|= |j/J|=O(1)$, there exists some $c>0$ such that
    \[|Ef_j|\gtrsim J^{-1}1_{R(v_j,0;cJ^2,cJ)},\]
    where $v_j=(-2j/J,1)/|(-2j/J,1)|$. Since all rectangles $R(v_j,0;cJ^2,cJ)$ contain a ball $B(0,cJ)$ around $0$, we conclude that
    \[\|E_hf\|_{L^q(\{|x\cdot \omega|\leq 1/2\})} \gtrsim \|J^{-\frac{1}{2}}1_{B(0,cJ)}(x)\|_{L^q(\{|x\cdot \omega|\leq 1/2\})} \gtrsim J^{\frac{1}{q}-\frac{1}{2}}.\]
    Therefore, if \eqref{e:resriction_strip} holds, we must have $J^{\frac{1}{q}-\frac{1}{2}}\lesssim_p 1$. Letting $J\to \infty$ implies the statement.
\end{proof}

\subsubsection{Necessity of $1/p+2/q\leq 1$ in the non-transversal case}
\begin{proof}
   The proof follows the same reasoning as in the transversal case, except that in this case we obtain a stronger lower bound than in \eqref{e:strip_cap_knapp_measure}, namely $\gtrsim \delta^{-2}$. 
\end{proof}

\subsubsection{Necessity of $p\leq q$ when $1/p+2/q=1$ in the non-transversal case}
\begin{proof} 
    Using parabolic rescaling \eqref{e:E_h_change_of_var}, we can reduce the problem to the proof of unboundedness of the operator $E_h$, defined in \eqref{e:E_h_def}, with $h(x)=x^2$, where $I\subset \R$ is a small interval around $0$ on a vertical strip. We can also assume that the interval has width $1$ to simplify the notation.

    Let $f_\eps$ be the same function as in the proof of Proposition \ref{p:Tlow}. Then Proposition \ref{p:Tlow} implies 
    \[|E_{h}f_{\eps}(0,x_2)|=|T_{\low}f_\eps(x_2)|\gtrsim |x_2|^{-\frac{1}{q}}\]
    Moreover, for $\abs{x_1}\ll 1$ it is natural to expect that $E_{h}f(x_1,x_2)\approx E_{h}f(0,x_2)$.

    Indeed, using the same notation as in the proof of Proposition \ref{p:Tlow}, we have
    \begin{align*}
      E_{h}(f_{\eps})(x_1,x_2)
      &=x_2^{-\frac{1}{q}}\int_{\eps \sqrt{x_2}}^{\eps^{-1} \sqrt{x_2}}t^{-\frac{1}{p}-1}\int_{0}^{\sqrt{x_2}} e^{-\pi \psi(t)\xi^2+2\pi i \xi x_1/\sqrt{x_2}}\, \d \xi\d t
    \end{align*}
    Recalling $\psi(t)=t^{-2}-2i$, we shall define 
    \[
    z_0:=z_0(x_1,x_2,t)=\frac{ix_1}{\sqrt{x_2}\psi(t)}.
    \]
    The inner integral is (after completing the square) equal to:
    \[e^{-\pi\frac{x_1^2}{x_2\psi(t)}}\int_{0}^{\sqrt{x_2}} e^{-\pi \psi(t)(\xi-z_0)^2 }\, \d \xi =
    e^{-\pi\frac{x_1^2}{x_2\psi (t)}}\int_{\gamma_1-z_0}e^{-\pi \psi(t)z^2} \, \d z,\]
    where $\gamma_1$ is the path in Figure \ref{f:contour}.
    Denote 
    \[F(t):=\int_{\gamma_1-z_0}e^{-\pi \psi(t)z^2}
    \, \d z.
    \]
    Our goal is to show that there exists $\delta,c_* >0$ such that the following inequality
    \begin{equation}
    \label{e:E_x2_lower}
        x_2^{\frac{1}{q}}|E_h(x_1,x_2)|=\Big|\int_{\eps \sqrt{x_2}}^{\eps^{-1} \sqrt{x_2}}t^{-\frac{1}{p}-1}e^{-\pi\frac{x_1^2}{x_2\psi(t)}}F(t)\d t\Big|\gtrsim 1
    \end{equation}
    holds for every $(x_1,x_2)\in [-\delta, \delta]\times [c_*,\frac{\eps^{-2}}{4}]$. 
    We claim that the statement will follow once we prove that the following estimates hold for  $(x_1,x_2)\in [-\delta, \delta]\times [c_*,\frac{\eps^{-2}}{4}]$:
    \begin{gather}
        \label{e:F_lower}
        \Re (F(t))\gtrsim \frac{1}{|\psi(t)|^{\frac{1}{2}}}, \\
        \label{e:F_upper}
        \Bigl|\int_{\eps \sqrt{x_2}}^{\eps^{-1} \sqrt{x_2}}t^{-\frac{1}{p}-1}(e^{-\pi\frac{x_1^2}{x_2\psi(t)}}-1)F(t)\d t\Bigr|\lesssim \delta^2.
    \end{gather}
    Indeed, using triangle inequality and estimates \eqref{e:F_lower}, \eqref{e:F_upper}, it follows
    \begin{align*}
    &\Big|\int_{\eps \sqrt{x_2}}^{\eps^{-1} \sqrt{x_2}}t^{-\frac{1}{p}-1}e^{\pi\frac{b^2}{4\psi (t)}}F(t)\d t\Big|
    \ge \Big| \int_{\eps \sqrt{x_2}}^{\eps^{-1} \sqrt{x_2}} t^{-\frac{1}{p}-1}F(t)\d t \Big|-\Big|\int_{\eps \sqrt{x_2}}^{\eps^{-1} \sqrt{x_2}}t^{-\frac{1}{p}-1}(e^{-\pi\frac{x_1^2}{x_2\psi(t)}}-1)F(t)\d t\Big|
    \\
    &\quad \ge \Re\Bigl( \int_{\eps \sqrt{x_2}}^{\eps^{-1} \sqrt{x_2}} t^{-\frac{1}{p}-1}F(t)\d t \Bigr)-C_1\delta^2
    \ge C_2\int_{\eps \sqrt{x_2}}^{\eps^{-1} \sqrt{x_2}} t^{-\frac{1}{p}-1}\frac{1}{|\psi(t)|^{\frac{1}{2}}}\d t - \delta^2 \ge C_2- C_1\delta^2.
    \end{align*}
    When $\delta>0$ is small enough with respect to $C_1,C_2$, the inequality \eqref{e:E_x2_lower} holds in the required range.

    We first need some estimates on $F$. By the change of contours observe that
    \begin{equation}
        \label{e:F_diff_gamma_1}
        \Big|F(t)-\int_{\gamma_1}e^{-\pi \psi(t)z^2}dz\Big|\leq \sum_{j=1}^{2} \Big|\int_{\rho_j(z_0)}e^{-\pi\psi(t)z^2}\, \d z \Big|,
    \end{equation}
    where for $r\in[0,1]$ the path $\rho_1(z)=rz_0$ is a segment from $0$ to $z_0$ and $\rho_2(z)=\sqrt{x_2}+rz_0$ is a segment from $\sqrt{x_2}$ to $\sqrt{x_2}+z_0$ (see also Figure~\ref{f:contour}).

    Using triangle inequality and \eqref{e:int_gamma2}, for $x_2\ge 1$ it holds that
    \begin{align}
    \label{e:gamma_1_abs_bound}
      \Big|\int_{\gamma_1}e^{-\pi \psi(t)z^2}dz \Big|
      &\leq 
      \Big|\int_{-\gamma_3}e^{-\pi \psi(t)z^2}dz \Big|
      + 
      \Big|\int_{\gamma_2}e^{-\pi \psi(t)z^2}dz \Big|
      \lesssim (1+x_2^{-1})|\psi(t)|^{-\frac{1}{2}} \lesssim |\psi(t)|^{-\frac{1}{2}}.
    \end{align}

    We continue by bounding the integrals over paths $\rho_j$.
    Observe now that for $x_1\in [-\delta, \delta]$ and $x_2\ge 1$, the following inequality holds for $j=1,2$
    \begin{align*}
        \Big|\int_{\rho_j(z_0)}e^{-\pi\psi(t)z^2}\, \d z\Big|\leq |z_0|\cdot \max_{z\in \rho_{j}(z_0)} |e^{-\pi \psi(t)z^2}|\leq \frac{\delta}{|\psi(t)|}\max_{z\in \rho_{j}(z_0)} |e^{-\pi \psi(t)z^2}|.
    \end{align*}
    It remains to bound
    \[
    \max_{z\in \rho_{j}(z_0)} |e^{-\pi \psi(t)z^2}|, \quad j=1,2.
    \]
    When $j=1$, observe that
    \begin{equation}
        \label{e:exponent_bound}
        |\pi \psi(t) (rz_0)^2| \leq \frac{\delta^2}{|\psi(t)|}\lesssim \delta^2, \quad r\in [0,1],
    \end{equation}
    so we conclude that 
    \[\max_{z\in \rho_{j}(z_0)} |e^{-\pi \psi(t)z^2}| \lesssim e^{C\delta^2}.\]
    On the other hand, for $j=2$, we have
    \[e^{-\pi \psi(t)(\sqrt{x_2}+rz_0)^2}= e^{-\pi \psi(t)x_2}\cdot e^{-2\pi \psi(t) \sqrt{x_2}rz_0}\cdot e^{-\pi \psi(t)r^2z_0^2}.\]
    The last factor in the product is bounded by $e^{C\delta^2}$ by the $j=1$ case. The middle factor is exponential of an imaginary number so it is absolutely bounded by $1$, and finally, the first factor is absolutely bounded by $1$ because $\Re(-\pi \psi(t)x_2)<0$ whenever $t,x_2>0$.
    Therefore, we obtain
    \begin{equation}
        \label{e:rho_bounds}
        \Big|\int_{\rho_j(z_0)}e^{-\pi\psi(t)z^2}\, \d z \Big|\lesssim  \frac{\delta e^{C\delta^2}}{|\psi(t)|}, \quad j=1,2.
    \end{equation}

    Using traingle inequality and estimates \eqref{e:F_diff_gamma_1}, \eqref{e:gamma_1_abs_bound} and \eqref{e:rho_bounds}, we conclude that for $\delta \leq 1$ the following inequality holds
    \[|F(t)|\lesssim \frac{1}{|\psi(t)|^{\frac{1}{2}}}+\frac{\delta e^{C\delta^2}}{|\psi(t)|}\lesssim 1.\]

    We turn to the proof of inequality \eqref{e:F_upper}. Because of estimate \eqref{e:exponent_bound}, we may invoke an elementary inequality $|e^z-1|\lesssim |z|$ (valid for $|z|\leq 1$) to deduce that the following inequality holds because of the triangle inequality and the estimate \eqref{e:exponent_bound}:
    \begin{align*}
        &\Bigl|\int_{\eps \sqrt{x_2}}^{\eps^{-1} \sqrt{x_2}}t^{-\frac{1}{p}-1}(e^{-\pi\frac{x_1^2}{x_2\psi(t)}}-1)F(t)\d t\Bigr|
        \lesssim \int_{\eps \sqrt{x_2}}^{\eps^{-1} \sqrt{x_2}}t^{-\frac{1}{p}-1}\big|(e^{-\pi\frac{x_1^2}{x_2\psi(t)}}-1)\big|\d t\\
        &\quad \lesssim \int_{0}^{\infty}t^{-\frac{1}{p}-1}\pi\Big|\frac{x_1^2}{x_2\psi(t)}\Big| dt
        \lesssim \delta^2 \int_{0}^{\infty}t^{-\frac{1}{p}-1} \frac{t^2}{(1+4t^4)^{1/2}}dt
        \lesssim \delta^2.
    \end{align*}
    
    Therefore, it remains to prove inequality \eqref{e:F_lower}. Using inequality \eqref{e:F_diff_gamma_1} and \eqref{e:rho_bounds} and \eqref{e:int_gamma1}, for $x_2\ge c_*$ (where $c_*$ is defined in the proof of Proposition \ref{p:Tlow}), the following inequality holds
    \begin{align*}
      \Re F(t) &\ge \Re (\int_{\gamma_1}e^{-\pi \psi(t)z^2}\, \d z) - C_1 \frac{\delta e^{C\delta^2}}{|\psi(t)|}
      \ge C_2 \frac{1}{|\psi(t)|^{\frac{1}{2}}}-C_1 \frac{\delta e^{C\delta^2}}{|\psi(t)|^{\frac{1}{2}}}.
    \end{align*}
    Choosing $\delta>0$ small enough, we conclude the estimate \eqref{e:F_lower}.    
\end{proof}

\section{Proof of Theorem~\ref{t:parabola}}
\label{s:proof_parabola}

\subsection{Transversal case}
The transversal case of estimate \eqref{e:radon_restr} is just the Hausdorff--Young inequality. Therefore, we proceed with the proof of estimate \eqref{e:resriction_strip}.
\begin{proof}[Proof of estimate \eqref{e:resriction_strip} in the transversal case]
    Observe that due to the fact that we use the pushforwad measure, the estimate is equivalent to
    \[\norm{Ef}_{L^q(\R\times I)} \lesssim \norm{f}_{L^p(\R)},\]
    where the operator $E$ was defined in \eqref{e:E_def} and $I=[-\frac{1}{2},\frac{1}{2}]$. By Plancherel's theorem, for any $x_2\in \R$ fixed, we obtain
    \[\norm{Ef(x_1,x_2)}_{L^2_{x_1}(\R)} = \norm{(e^{ix_2\xi^2}f(\xi))^{\vee}(x_1)}_{L^2_{x_1}(\R)} =\norm{f}_{L^2(\R)},
    \]
    so the $L^2\to L^2$ boundedness follows by integration in $x_2\in I$. Finally, the statement follows by complex interpolation with Theorem \ref{e:restriction_2d}.
    
\end{proof}

\subsection{Non-transversal case}
The non-transversal case of estimate \eqref{e:radon_restr}, using \eqref{e:radon_change_var}, is reduced to Proposition \ref{p:T}, so we proceed with the proof of estimate \eqref{e:resriction_strip}.
\begin{proof}[Proof of estimate \eqref{e:resriction_strip} in the non-transversal case]
    Observe that by using parabolic rescaling \eqref{e:E_h_change_of_var} and a change of variables, we can reduce the general case to $\omega^{\bot}=(0,1)$. This change of variables alters the width of the strip, but since we can cover a wider strip with $O_\omega(1)$ thinner strips, we can ignore this issue. 

    When $\omega=(1,0)$, the problem reduces to the following estimate
    \[\|Ef\|_{L^q(I\times\R)}\lesssim_{p,q} \norm{f}_{L^p},\]
    where the operator $E$ is defined in \eqref{e:E_def} and $I=[-\frac{1}{2},\frac{1}{2}]$.

We decompose the operator $E$ into low and high frequency part:
\[E_\low f(x_1,x_2):=E (1_{[-1,1]}f),\quad E_{\high}=E-E_\low.\]
From Theorem \ref{t:compact} it follows that $E_\low$ is bounded whenever 
\begin{equation}
    \label{e:E_low_region}
    (1/p,1/q) \in \triangle OAD \setminus [DE).
\end{equation}
On the other hand, Proposition \ref{p:Thigh} implies that the estimate 
\[\norm{E_\high f(x_1,x_2)}_{L^q_{x_2}(\R)}\lesssim \norm{f}_{L^p(\R)}\]
holds uniformly in $x_1\in \R$ whenever 
\[(1/p,1/q)\in  \triangle ACD\setminus ([A,C]\cup [C,D] \cup [D,E)).\] 
Integrating this in $x_1\in I$ yields
\begin{equation}
    \label{e:E_high}
    \norm{E_\high f}_{L^q(I\times \R)}\lesssim \norm{f}_{L^p(\R)}
\end{equation}
for $(p,q)$ in the same range. Complex interpolation with the restriction theorem (Theorem~\ref{t:restriction_2d}) implies that the estimate \eqref{e:E_high} holds\footnotemark whenever 
\begin{equation}
    \label{e:E_high_region}
    (1/p,1/q) \in \square ACDB \setminus ([A,C]\cup [C,D]\cup [BD] ).
\end{equation}
Therefore, the operator $E$ is bounded on the intersection of regions \eqref{e:E_low_region} and \eqref{e:E_high_region}, and the statement follows.

\footnotetext{Although the full operator $E$ is unbounded on $\triangle ACD\setminus [A,E]$, it is nevertheless useful to show that the high-frequency piece $E_{\high}$ (or $T_\high$) in $\triangle ACD\setminus ([A,C]\cup [C,D]\cup [BD] )$. }
\end{proof}

\subsection{Necessity of the conditions}
\subsubsection{Necessity of $1/p+3/q\ge 1$ in both cases}
\begin{proof}
    We first show that the condition $\frac1p+\frac3q\geq 1$
    is a necessary condition both for transversal and non-transversal case. 
    This is again a classical Knapp-type example. Let $\phi\in C_c^{\infty}$ be a function such that $\phi(0)=1$ and let $f_R=\phi(\frac{\cdot}{R})$. Then $\|f_R\|_{L^p}\sim R^\frac1p$. 
    
    On the other hand, Lemma \ref{l:knapp} gives 
    \[|Ef(x)|\gtrsim R 1_{[0,1/R]\times[0,1/R^2]}(x).\] Consequently, 
    \[
    \|Ef\|_{L^q(\R\times I)} 
    \geq 
    \|Ef\|_{L^q([0,1/R]\times[0,1/R^2])}
    \gtrsim
    R^{1-\frac 3q}.
    \]
    Letting $R\to\infty$ gives the conclusion.
\end{proof}

\subsubsection{Necessity of $1/p+1/q=1$ and $p\leq q$ in the transversal case, Radon transform}
\begin{proof}
    Necessity of the condition $\frac{1}{p}+\frac{1}{q}=1$ and $p\leq q$ for the estimate \eqref{e:radon_restr} follows from the fact that, with $x_2$ fixed,
    \[E f(x_1,x_2)= (e^{2\pi i\xi^2 x_2}f(\xi))^{\wedge}(x_1) \]
    and necessary conditions for the Fourier transform.
\end{proof}

\subsubsection{Necessity of $1/p+1/q\leq 1$ and $q\ge 2$ in the transversal case}
This follows from the necessity of the same conditions in Theorem~\ref{t:compact} (see Sections~\ref{sss:antidiagonal} and \ref{sss:q>=2}).

\subsubsection{Necessity of $p,q=2$ if $p=q$ in the transversal case}
We prove that the estimate \eqref{e:resriction_strip} cannot hold in transversal case for any $p=q\in (2,4]$. 
The proof is inspired by \cite{BCSS89}.
\begin{proof}
    Assume on the contrary that the estimate holds for some $p=q\in (2,4]$. 

    Let $J_0\in \N$ and $J:=2^{J_0}$. Let $\phi\in C_c^{\infty}$ be a function with support in $[-0.1,0.1]$ and $\phi(0)=1$.
    Define
    \[
    f
    =
    \sum_{j=J}^{2J-1} \epsilon_j f_j, 
    \]
    where $\epsilon_j\in\{\pm 1\}$ will be chosen shortly and
    \[
    f_j(\xi)
    =
    e^{2\pi i (a_{j,1} \xi + a_{j,2}\xi^2)}\tau_{j+\frac{1}{2}}\phi(\xi)
    \]
    for $a_j=(a_{j,1},a_{j,2})\in \R^2$ which we will choose later. 
    
    For arbitrary choice of $(\epsilon_j)_j$, note that 
    \begin{equation}\label{e:LHS Kakeya-type}
    \norm{f}_{L^p} 
    = 
    \Big(\sum_{j=J}^{2J-1} \norm{f_j}_{L^p}^p \Big)^{\frac1p} 
    \sim 
    J^{\frac1p}
    \end{equation}
    since $(\supp f_j)_j$ are mutually disjoint.  
    
    On the other hand, by Khintchine's inequality we have
    \[
    \mathbb{E}_{(\epsilon_j)} \Big\|\sum_{j=J}^{2J-1} \epsilon_jEf_j\Big\|_{L^p(\R\times I)}^p
    \gtrsim_p
    \Big\|\Big(\sum_{j=J}^{2J-1} |Ef_j|^2 \Big) ^\frac12\Big\|_{L^p(\R\times I)}^p, 
    \]
    so there exists a choice of signs $(\epsilon_j)_{j}$ such that
    \[\|E f\|_{L^p(\R\times I)} \gtrsim  \Big\|\Big(\sum_{j=J}^{2J-1} |Ef_j|^2 \Big) ^\frac12\Big\|_{L^p(\R\times I)}.
    \]
    Combining this with \eqref{e:LHS Kakeya-type} and by duality of $L^{p/2}$ norm, we obtain
    \begin{equation}\label{e:dual with K}
    \Big|\sum_{j=J}^{2J-1} \int_{\R\times I}  |E_j f_j(x)|^2 1_{K}(x)\,\d x\Big|
    \lesssim
    J^\frac2p |K|^{1-\frac2p}
    \end{equation}
    for any $K\subset \R\times I$. Therefore, application of Lemma \ref{l:knapp} further implies that 
    \begin{equation}\label{e:dual with tree}
    \Big|\sum_{j=J}^{2J}  \int_{\R\times i} 
    \one_{R(v_j,a_j;cj,c/j)}(x)
    1_{K}(x)\,\d x\Big|
    \lesssim
    J^\frac2p |K|^{1-\frac2p},
    \end{equation}
    where $v_j=(-2j-1,1)/|(-2j-1,1)|$.
    
    Let $A^{(0)}$ be the triangle with vertices  $(0,0)$, $(-2J,1)$, $(-4J,1)$. Partition the segment $[-4J,2J]$ into $J$ equal subsegments and note that the vectors $(v_j)_{j=J}^{2J-1}$ point toward the midpoint of each subsegment. Moreover, each triangle with vertices $(0,0)$, $(-2j,1)$ and $(-2j-2,1)$ contains a rectangle of dimensions $c_1J\times c_1/J$ oriented in the direction $v_j$, for some small $c_1>0$, since the angle at the point $(0,0)$ is $\sim J^{-2}$. We choose $c_1$ so that $c_1<c$.
    
    Then if we apply the $J_0$-th iteration of Perron tree construction applied to the triangle $A_0$,
    we arrive to the shape $A^{(J_0)}$ that contains a rectangle $R(v_j,b_j; c_1/J, c_1J)$ and has area at most 
    \[|A^{(J_0)}|\lesssim J \frac{\log\log J}{\log J}.\] 
    Set 
    $a_j:=b_j$ and define  $K:=\bigcup_{j=J}^{2J-1} R(v_j,b_j;c_1j,c_1/j)$.
    Since $c_1<c$ and $K\subset A^{(J_0)}$, invoking \eqref{e:dual with tree} gives
    \[
    J\sim \Big|\sum_{j=J}^{2J-1}  \int_{\R\times \mathbb{I}} 
    \one_{R(v_j,a_j;cj,c/j)}(x)
    1_{K}(x)\,\d x\Big|
    \lesssim
    J^\frac2p
    \Big(J \frac{\log\log J}{\log J}\Big)^{1-\frac2p}.
    \]
    Finally, letting $J\to\infty$, we conclude that $p\leq 2$ is necessary.
\end{proof}







\subsubsection{Necessity of $1/p+1/q\ge 1/2$ in the non-transversal case}
\begin{proof}
    Let $f=\sum_{j=0}^{J}\epsilon_jf_j$
    where $\epsilon_j\in \{\pm 1\}$ will be chosen shortly, and
    \[f_j(\xi):= \tau_j\phi(\xi)\] 
    Since $(\supp f_j)_j$ are disjoint,  
    \[\norm{f}_{L^p(\R)} = \left( \sum_{j=0}^{J} \norm{f_j}_{L^p(\R)}^p \right)^{1/p} \sim J^{1/p}.\]
    On the other hand, using Khintchine's inequality, we have
    \begin{equation*}
        \mathbb{E}_{(\epsilon_j)_j} 
        \Big(\Big\|\sum_{j=0}^{J} \epsilon_jEf_j(x,t)\Big\|_{L^p(\R)}^p \Big)
        \gtrsim_p \Big\|\Big( \sum_{j=0}^{J} |Ef_j(x,t)|^2 \Big)^{1/2}\Big\|_{L^p(\R)}^p.
    \end{equation*}
    we conclude that there exists some choice of $(\epsilon_j)_j$ such that
    \begin{equation}
        \label{e:khin_for_triangle}
        \|Ef\|_{L^p(\R)} \gtrsim_p \Big\|\Big( \sum_{j=0}^{J} |Ef_j(x,t)|^2 \Big)^{1/2}\Big\|_{L^p(\R)}.
    \end{equation}
    
    By Lemma \ref{l:knapp} we know that
    \[
    E\tau_j\phi(x)\gtrsim 1_{R(v_j,0;cj,c/j)}(x).
    \]
    where $v_j=(-2j,1)/|(-2j,1)|$. Note that 
    \begin{align*}
        \bigcap_{j=0}^J R(v_j,0,cj,c/j)
        =
        R(v_J,0,cJ,c/J) \cap \{(x_1,x_2)\in\R^2: -c\leq x_1\leq c\}=:A,
    \end{align*}
    and that $A$ is a parallelogram with height $c$ and base $\sim J^{-1}$. Therefore, using
    \[\Big(\sum_{j=0}^{J} |Ef_j(x,t)|^2 \Big)^{1/2} \gtrsim J^{\frac{1}{2}}1_{A}(x),\]
    the estimate \eqref{e:khin_for_triangle} implies
    \[\|Ef\|_{L^p}\gtrsim J^{1/2}|A|^{\frac{1}{q}} \sim J^{\frac{1}{2}-\frac{1}{q}}. \]
    Letting $J\to\infty$, we reach the conclusion.
\end{proof}
\subsubsection{Failure at $p=q=4$ in non-transversal case}
\begin{proof}
    Using parabolic rescaling \eqref{e:E_h_change_of_var}, we reduce the problem to the $\omega = (1,0)$ case.
    In this case, we need more precise estimates than the lower bound that follows from Knapp-type argument because it does not take the tail of the function into account.
    
    Let $0<\eps \ll1$. We test the operator on the function $f_\eps(\xi)=e^{-\pi \eps \xi^2}$. Observe that
    \[\norm{f_\eps}_{L^p(\R)}\sim \eps^{-\frac{1}{2p}}.\]
    Using \eqref{e:gauss int}, we have
    \[\abs{Ef(x_1,x_2)}^q = (\eps^2+4x_2^2)^{-\frac{q}{4}} e^{- q\frac{\pi \eps x_1^2}{\eps^2+4x_2^2}}.\]
    Therefore, using change of variables we obation
    \[ \|Ef\|_{L^q(I\times \R)}^q = \eps^{-\frac{q}{2}+\frac{3}{2}} \int_{\R}(1+4x_2^2)^{-\frac{q}{4}+\frac{1}{2}}\int_{-(\eps(1+4x_2^2))^{-1/2}}^{(\eps(1+4x_2^2))^{-1/2}} e^{-4\pi x_1^2}dx_1dx_2.\]
    Let $q=4$. When $\eps(1+4x_2^2)\leq 1$, the inner integral is greater than $\int_{-1}^{1}e^{-4\pi x^2}dx\gtrsim 1$, so when $x_2< \eps^{-\frac{1}{2}}/8$, the total expression is bounded from below with
    \[\eps^{-\frac{1}{2}}\int_{-\eps^{-\frac{1}{2}}/8}^{\eps^{-\frac{1}{2}}/8}(1+4x_2^2)^{-\frac{1}{2}}dx_2 \gtrsim \eps^{-\frac{1}{2}}\log \eps^{-1}.\]
    Therefore, if $(4,4)$ estimate was true, one would have
    \[\eps^{-\frac{1}{8}}(\log \eps^{-1})^{\frac{1}{4}}\lesssim \eps^{-\frac{1}{8}},\]
    what gives a contradiction when $\eps\to 0$.
\end{proof}

\section*{Acknowledgments}
The authors thank Rudi Mrazović for helpful comments on an earlier version of this manuscript.

This work was supported in part by the \emph{Croatian Science Foundation} project IP-2022-10-5116 (FANAP).


\bibliography{bibliography}
\bibliographystyle{plain}

\end{document}